\documentclass[a4paper,11pt]{amsart}
\usepackage[pdftex]{color,graphicx}
\usepackage{amssymb}
\usepackage{amsfonts}
\usepackage{esint}
\usepackage{amsmath}
\usepackage{amsrefs}
\usepackage{epsfig}
\usepackage{amsthm}
\usepackage{color}
\usepackage[]{epsfig}
\usepackage[]{pstricks}
\usepackage{tikz}
\usetikzlibrary{calc,through,intersections,patterns,patterns.meta,arrows.meta}
\usepackage{caption}
\usepackage[utf8]{inputenc}
\newtheorem{theorem}{Theorem}[section]

\newtheorem{lemma}{Lemma}[section]
\newtheorem{corollary}{Corollary}[section]

\theoremstyle{definition}
\newtheorem{definition}{Definition}[section]
\newtheorem{example}{Example}[section]

\theoremstyle{remark}
\newtheorem{remark}{Remark}[section]

\newcommand{\R}{\mathbb{R}}

\newcommand{\DD}{\mathcal{D}}
\newcommand{\PP}{\mathcal{P}}
\newcommand{\CC}{\mathcal{C}}
\newcommand{\LL}{\mathcal{L}}
\newcommand{\HH}{\mathcal{H}}
\newcommand{\VV}{\mathcal{V}}

\newcommand{\n}{\nabla}
\newcommand{\ran}{\rangle}
\newcommand{\lan}{\langle}
\newcommand{\ve}{\varepsilon}
\newcommand{\vp}{\varphi}
\newcommand{\sss}{\sigma}

\DeclareMathOperator{\hess}{Hess}
\DeclareMathOperator{\hesss}{hess}

\DeclareMathOperator{\tr}{trace}

\DeclareMathOperator{\dist}{dist}

\numberwithin{equation}{section}
\title[Half space theorems for the $r$-mean curvature flow]{Half-space theorems for translating solitons of the $r$-mean curvature flow}
\author{Hil\'ario Alencar, G. Pacelli Bessa \and Greg\'orio Silva Neto}
\date{\today}

\address{Instituto de Matem\'atica,
Universidade Federal de Alagoas,
Macei\'o, 57072-900, Brazil}
\email{hilario@mat.ufal.br}

\address{Departamento de Matem\'atica,
Universidade Federal do Ceará,
Fortaleza, 60455-760, Brazil}
\email{bessa@mat.ufc.br}

\address{Instituto de Matem\'atica,
Universidade Federal de Alagoas,
Macei\'o, 57072-900, Brazil}
\email{gregorio@im.ufal.br}

\begin{document}

\begin{abstract}
In this paper, we establish nonexistence results for complete translating solitons of the $r$-mean curvature flow under suitable growth conditions on the $(r-1)$-mean curvature and on the norm of the second fundamental form. We first show that such solitons cannot be entirely contained in the complement of a right rotational cone whose axis of symmetry is aligned with the translation direction. We then relax the growth condition on the $(r-1)$-mean curvature and prove that properly immersed translating solitons cannot be confined to certain half-spaces opposite to the translation direction. We conclude the paper by showing that complete, properly immersed translating solitons satisfying appropriate growth conditions on the $(r-1)$-mean curvature cannot lie completely within the intersection of two transversal vertical half-spaces.
\end{abstract}

\maketitle

\section{Introduction}
Let $X_0:\Sigma^n\to \R^{n+1}$ be an isometric immersion of a $n$-dimensional Riemannian manifold $\Sigma^n,$ with the second fundamental form
\[
II(X,Y)=\lan A(X),Y\ran N,
\]
where $A:T\Sigma^n\to T\Sigma^n$ is its shape operator and $N$ is a unit normal vector field. Letting $k_1,\ldots,k_n$ be the principal curvatures of the immersion, we define the $r$-mean curvatures as
\begin{equation}\label{eq1.1}
 \left\{\begin{aligned}
 \sss_0&=1,\\  \sss_r&=\displaystyle{\sum_{i_1<\cdots<i_r}k_{i_1}\cdots k_{i_r}},\quad \mbox{for}\quad 1\leq r\leq n,\\
 \sss_r&=0,\,\,\mbox{for}\,\, r>n,\\
\end{aligned}\right.
\end{equation}
where $(i_1,\ldots,i_r)\in\{1,\ldots,n\}^r.$ These functions appear naturally in the characteristic polynomial of $A,$ since
\begin{equation}\label{char}
\begin{aligned}
\det(A-tI)&= \sss_n-\sss_{n-1}t+\sss_{n-2}t^2 -\sss_{n-3}t^{3}+\cdots +(-1)^nt^n \\
&= \sum_{j=0}^n(-1)^j\sss_{n-j}t^j. 
\end{aligned}
\end{equation}
A family of isometric immersions $\mathcal{X}:\Sigma^n\times[0,T)\to\R^{n+1}$ is a solution of the $r$-mean curvature flow if it satisfies the initial value problem
\begin{equation}\label{eq-flow}
    \left\{\begin{aligned}
    \dfrac{\partial\mathcal{X}}{\partial t}(x,t)&=\sigma_r(k_1(x,t),\ldots,k_n(x,t))\cdot N(x,t),\\
        \mathcal{X}(x,0)&=X_0(x).\\
    \end{aligned}\right.
\end{equation}
Here, $k_1(x,t),\ldots,k_n(x,t)$ are the principal curvatures of the immersions $X_t:=\,\mathcal{X}(\cdot,t)$ and $N(\cdot,t)$ are their normal vector fields. When $r=1,$ the $r$-mean curvature flow is  called the mean curvature flow. 

The $r$-mean curvature flow has been widely studied in recent decades, as we can see, for instance, in \cite{AS2010}, \cite{AMC2012}, \cite{AW2021}, \cite{BS2018}, \cite{CRS2010}, \cite{Chow1987}, \cite{GLM2018}, \cite{GLW2017}, \cite{LWW2021}, \cite{LSW2020}, \cite{Urbas1990}, \cite{Urbas1999}, and \cite{Z2013}. 


Among the most significant solutions of \eqref{eq-flow} are the self-similar ones, in which the initial immersion $X_0$ evolves under the flow only through an isometry or a homothety in $\mathbb{R}^{n+1}$. Translating solitons, also called translators, form the subclass where the evolution consists solely of a translation in $\mathbb{R}^{n+1},$ i.e., 
\[
\mathcal{X}(x,t)=X_0(x)+tV,
\]
where $V\in\R^{n+1}$ is a unit vector, called the velocity vector of the translator. It can be easily proven that if $\Sigma^n\subset\R^{n+1}$ is a translating soliton of the $r$-mean curvature flow, then 
\begin{equation}\label{def-trans}
    \sigma_r(x)=\lan N(x),V\ran.
\end{equation}

A natural tool to study the $r$-mean curvature is the $r$-th Newton transformation $P_r:T\Sigma^n\to T\Sigma^n,$ $0\leq r\leq n-1,$ defined recursively by
\begin{equation}\label{Pr-defi}
\left\{\begin{aligned}
P_0&=I,\\
P_r&=\sigma_rI-AP_{r-1},\quad r\geq 1,\\
\end{aligned}\right.
\end{equation}
where $I:T\Sigma^n\to T\Sigma^n$ is the identity operator. In the context of Differential Geometry, it first appeared in the work of Reilly \cite{Reilly} in the expressions of the variational integral formulas for functions  $f(\sss_0,\ldots, \sss_n)$ of the elementary symmetric functions $\sss_i$'s. Since then, the Newton transformations have been widely used as a tool in the study of the $r$-mean curvature, as we can see, for example, in \cite{ABS-2023}, \cite{ABS-2025}, \cite{AdCE-2003}, \cite{AdCR-1993}, \cite{AIR-2013}, \cite{BX-2024}, \cite{BMR-2013}, \cite{CY-1977}, \cite{HL-1995}, \cite{IMR-2011}, \cite{R-1993}, and \cite{W-1985}. Since we are assuming that $\Sigma^n$ has a global choice of $N$ we have that $P_r$ is globally defined.

\begin{remark}
    Solving the recurrence, we can also write $P_r$ as the polynomial
\begin{equation}\label{Pr-pol}
    P_r = \sss_rI-\sss_{r-1}A+\sss_{r-2}A^2 -\cdots +(-1)^rA^r = \sum_{j=0}^r(-1)^j\sss_{r-j}A^j.
\end{equation}
Notice that the Newton transformation $P_r$ is a ``degree $r$ version'' of the characteristic polynomial \eqref{char} applied to $A.$
\end{remark}

In this paper, we establish non-existence results for translating solitons of the $r$-mean curvature flow in certain regions of $\R^{n+1}$. We first prove a non-existence theorem for translating solitons of the $r$-mean curvature flow in the complement of the open cone
\[
\CC_{V,a}=\left\{X\in\R^{n+1}; \,\left\lan \frac{X}{\|X\|},V\right\ran > \,a, \,\,a\in (0,1)\right\}.
\]
\begin{theorem}\label{thm-Cone}
There are no complete, $n$-dimensional, translating solitons of the $r$-mean curvature flow $\Sigma^n\subset\R^{n+1}$ with velocity $V$ with
$P_{r-1}$ positive semidefinite, contained in the complement of the open cone $\CC_{V,a}$,
\[
\left(\CC_{V,a}\right)^{c}=\left\{X\in\R^{n+1}; \,\left\lan \frac{X}{\|X\|},V\right\ran \leq \,a, \,\,a\in (0,1)\right\},
\]
 satisfying one of the following conditions:
\begin{itemize}
    \item[(i)] $\Sigma^n$ is properly immersed and
    \begin{equation}\label{HS2-hyp-1}\limsup_{\delta(x)\to\infty}\frac{\sigma_{r-1}(x)}{\delta(x)}<\frac{r(1-a)}{a(n-r+1)},
    \end{equation}
where $\delta(x)$ denotes the extrinsic distance to a fixed point of $\R^{n+1};$
    \item[(ii)] $\sigma_{r-1}$ is bounded and
    \begin{equation}\label{HS2-hyp-2}\limsup_{\rho(x)\to\infty}\frac{\|A(x)\|}{\rho(x)\log(\rho(x))\log(\log(\rho(x)))}<\infty,
    \end{equation}
    where $\rho(x)$ denotes the intrinsic distance to a fixed point of $\R^{n+1}.$
\end{itemize}
\end{theorem}

If $r=1,$ then $\sigma_{r-1}=\sigma_0=1$ and $P_{r-1}=P_0=I$ (that is positive definite). Moreover, \eqref{HS2-hyp-1} is automatically satisfied. Therefore, we immediately obtain:

\begin{corollary}\label{KP2}
    There are no complete, properly immersed, $n$-dimensional, translating solitons of the mean curvature flow $\Sigma^n\subset\R^{n+1},$ with velocity $V,$ contained in the complement of the open cone $\CC_{V,a}$,
\[\left(\CC_{V,a}\right)^{c} = \left\{ X \in \mathbb{R}^{n+1};\left\langle \frac{X}{\|X\|}, V \right\rangle \leq a,\,\, a\in (0,1) \right\}.\]
\end{corollary}

\begin{remark}\label{rem-change-hyp}
Clearly, condition (ii) in Theorem \ref{thm-Cone} can be used to replace the hypothesis of being properly immersed in Corollary \ref{KP2} by the control of the second fundamental form given by \eqref{HS2-hyp-2}, since, for $r=1,$ $\sigma_{r-1}=1$.
\end{remark}

\begin{remark}\label{control}
  Notice that the conditions given in Equations \eqref{HS2-hyp-1} and \eqref{HS2-hyp-2} are equivalent to the existence of positive constants $C,D>0$ such that
    \[
    \sigma_{r-1}(x)<\frac{r(1-a)}{a(n-r+1)}[\delta(x)+C] 
    \]
    and
    \[\|A(x)\|\leq D\rho(x)\log(\rho(x))\log(\log(\rho(x))),
    \]
    for $\delta(x)\gg1$ and $\rho(x)\gg1$ respectively.
\end{remark}

If we restrict the region $(\CC_{V,a})^c$ to a halfspace of the form
\[
\mathcal{H}_W=\left\{X\in\R^{n+1}; \lan X,W\ran \leq 0,\ \lan V,W\ran>0,\ \|W\|=1\right\},
\] 
that is always contained in $(\CC_{V,a})^c$ for any $a\in(0,1)$ such that $a\geq\lan V,W\ran,$ we can improve the hypothesis \eqref{HS2-hyp-1} in the case that $\Sigma^n$ is properly immersed. This is the content of the next

\begin{theorem}\label{thm-HS-1}
There is no complete, $n$-dimensional, properly immersed, translating soliton of the $r$-mean curvature flow $\Sigma^n\subset\R^{n+1},$ with velocity $V$, $P_{r-1}$ positive semidefinite, contained in the closed half-space 
\[
\mathcal{H}_W=\left\{X\in\R^{n+1}; \lan X,W\ran \leq 0,\ \lan V,W\ran>0,\ \|W\|=1\right\},
\] 
and such that
    \begin{equation}\label{HS1-hyp-1}
\limsup_{\delta(x)\to\infty}\frac{\sigma_{r-1}(x)}{[\delta(x)]^2\log(\delta(x))\log(\log(\delta(x)))}<\infty,
    \end{equation}
   where $\delta(x)$ denotes the extrinsic distance to a fixed point of $\R^{n+1}$.
\end{theorem}

\begin{remark}\label{control-2}
    Notice that the condition given in Equation \eqref{HS1-hyp-1} is equivalent to the existence of a positive constant $D>0$ such that
    \[
    \sigma_{r-1}(x)\leq D[\delta(x)]^2\log(\delta(x))\log(\log(\delta(x))) 
    \]
    for $\delta(x)\gg1.$ This growth condition for $\sigma_{r-1}(x)$ is clearly better than that given by Equation \eqref{HS2-hyp-1} in Theorem \ref{thm-Cone}.
\end{remark}

If $r=1,$ then $\sigma_{r-1}=\sigma_0=1$ and \eqref{HS1-hyp-1} is automatically satisfied. Thus, we obtain the following

\begin{corollary}\label{KP1}
    There are no complete, properly immersed, $n$-dimensional, translating solitons of the mean curvature flow $\Sigma^n\subset\R^{n+1},$ with velocity $V,$ in a closed half-space 
\[
\mathcal{H}_W=\left\{X\in\R^{n+1}; \lan X,W\ran \leq 0,\ \lan V,W\ran>0,\ \|W\|=1\right\}.
\] 
\end{corollary}


\begin{remark}
Corollaries \ref{KP1} and \ref{KP2} are stated in \cite{KP} without the hypothesis that $\Sigma^n$ is properly immersed. Unfortunately, there was an error in their proof. In order to apply the Omori–Yau maximum principle, the authors claimed that translating solitons of the mean curvature flow have Ricci curvature bounded below by $-1/4$ (Equation (3.1), p. 5), but, in their argument, the sign in the Gauss equation is reversed, leading to an incorrect estimate.
\end{remark}

\begin{example}\label{Ex-r1}
    For $r=1$, the Grim Reaper cylinder, the bowl soliton, and the translating catenoids are examples of properly immersed translating solitons that are not contained in $\mathcal{H}_W$ or  $\left(\CC_{V,a}\right)^{c}$ for $V=E_{n+1}$. Moreover, this property persists under any translation in the direction of $V$, i.e., for $\Sigma^n+tV$, $t\in\mathbb{R}$. For the Grim Reaper cylinder, this follows from the fact that the curve $y=-\log(\cos x)$ has two vertical asymptotes, implying that its graph intersects every rotational cone with vertex at the origin and axis $E_{n+1}$. By the same reason, no vertical translation in the $E_{n+1}$-direction can place it entirely within any $\mathcal{H}_W$. On the other hand, one can always find a suitable translation such that these hypersurfaces lie in the complements $\mathbb{R}^{n+1}\setminus\mathcal{H}_W$ and $\CC_{V,a}$ (see Figure \ref{fig:2}). The same phenomenon occurs for the bowl soliton and the translating catenoids, whose asymptotic expansion as $R(x)$ approaches infinity is
    \[
    \frac{R(x)^2}{2(n-1)}-\log(R(x))+O(R(x)^{-1}),
    \]
    where $R(x)$ denotes the Euclidean distance from $x\in\R^n$ to the origin (remember these hypersurfaces are radial graphs over a subset of $\R^n$).
    \begin{figure}[ht]
\centering
\def\r{5}
	\begin{tikzpicture}[scale=1,>=stealth]
		\def\u{8}
		\def\v{5}
		\coordinate (p1) at (60:8cm);
		\coordinate (p2) at (120:8cm);
		\clip (-4.5,5)--(-4.5,-2)--(4.5,-2)--(4.5,5)--cycle;
		\fill[black!10] (p1) -- (0,0) -- (p2) -- cycle;
		\fill[black!10] (-5,6.9)--(-5,-5)--(5,-5)--(5,6.9)--cycle;
		\fill[white] (p1)--(0,0)--(p2);
		\fill[black!50,draw] (135:1.5*\v cm) -- (0,0) -- (-45:\v cm)  -- (-135:1.5*\v cm)-- cycle;;
		
		
		\path[draw,-{Stealth[length=2mm]}] (-0.6*\u,0) -- (0.6*\u,0) node[pos=1, below] {};
		\path[draw,-{Stealth[length=2mm]}] (0,-\v) -- (0,0.9*\v) node[pos=1, right] {$V=E_{n+1}$};
		
		\path[draw,dashed] (0.4*\v,-\v) -- (0.4*\v,\v);
		\path[draw,dashed] (-0.4*\v,-\v) -- (-0.4*\v,\v);
		
		\node[below left] (k1) at (0.4*\v,0) {$\dfrac{\pi}{2}$};
		\node[below left] (k2) at (-0.4*\v,0) {$-\dfrac{\pi}{2}$};
		\node[] (k3) at (-0.8*\v,1) {$\mathcal{H}_{W}$};
		\node[] (k4) at (0.7*\v,1) {$\left(\CC_{V,a}\right)^{c}$};

		\draw (0,0) to[out=0,in=270] (0.4*\v,\v);
		\draw (0,0) to[out=180,in=270] (-0.4*\v,\v);
		
		
		\def\t{2}
		\coordinate (w1) at (40:1*\t cm);
		\coordinate (w2) at (-60:1.6*\t cm);
		
		\path[draw,-{Stealth[length=2mm]}] (60:0) -- (w1) node[pos=1, above] {$W$};
		
		
		
		\draw (p1) -- (0,0) (p2) -- (0,0);
		
		
		
	\end{tikzpicture}
     \caption{The Grim Reaper cylinder and the sets $\mathcal{H}_{V}$ and $\left( \CC_{V,a}\right)^{c}.$}
    \label{fig:2}
    
    \end{figure}
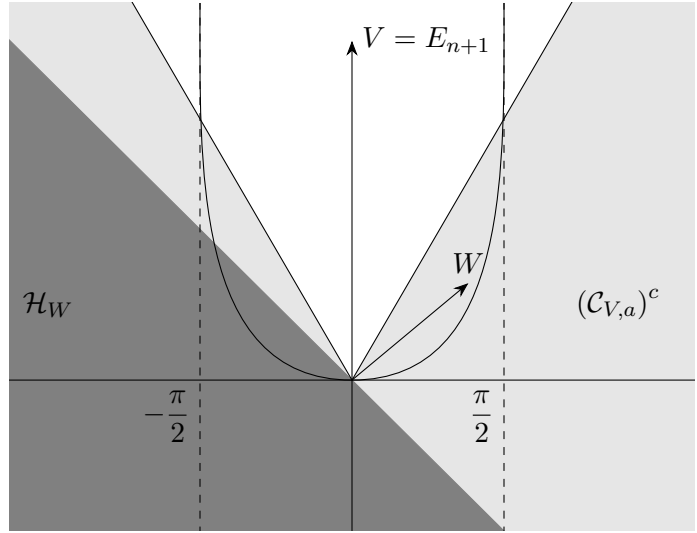

\end{example}

\begin{example}\label{Ex-r2}
    If $r>1$, the $r$-bowl soliton and the $r$-translating catenoids classified by R. de Lima and G. Pipoli in \cite{LP} share the same properties as their counterparts for $r=1.$ Indeed, they are properly embedded (despite, for $r>1,$ the $r$-translating catenoids are not complete), and neither they nor any of their translations in the direction of $V=E_{n+1}$ are contained in any $\mathcal{H}_W$ or $(\CC_{V,a})^c$. Moreover, one can always find a suitable translation such that these hypersurfaces lie in the complements $\mathbb{R}^{n+1}\setminus\mathcal{H}_W$ and $\CC_{V,a}$. In fact, de Lima and Pipoli proved that the angle function $\Theta=\langle N,E_{n+1}\rangle$ converges to $1$ as the Euclidean distance $R(x)$ from $x\in\mathbb{R}^n$ to the origin tends to infinity. It follows that these hypersurfaces intersect every rotational cone with fixed angle, and thus cannot be contained in any $\left(\CC_{V,a}\right)^{c}$ for $V=E_{n+1}$.
\end{example}

In order to introduce our next result, we need the following
\begin{definition}
    Let $\Sigma^n\subset\R^{n+1}$ be a translating soliton of the $r$-mean curvature flow with velocity vector $V$, this is, $\sigma_r=\lan N,V\ran,$ where $N$ is the normal vector field of $\Sigma^n$.
    \begin{itemize}
        \item[(i)] Given $B,W\in\R^{n+1},$ we say that a halfspace 
    \[
    \HH:=\HH_{(B,W)}:=\{X\in\R^{n+1};\lan X-B,W\ran\geq0\},
    \]
    is \emph{vertical} if $W\perp V;$ 
        \item[(ii)] Two halfspaces 
        \[
    \HH_i:=\HH_{(B_i,W_i)}:=\{X\in\R^{n+1};\lan X-B_i,W_i\ran\geq0\},\quad i=1,2,
        \]
        are \emph{transversal} if $W_1$ and $W_2$ are linearly independent.
    \end{itemize}
\end{definition}
We conclude this paper  proving a nonexistence result for translating solitons of the $r$-mean curvature flow in  intersection of two vertical half-spaces. This result is a generalization of the \textquotedblleft bi-halfspace\textquotedblright \,\,Theorem 1.1 of \cite{CM} by Chini and M\o ller.

\begin{theorem}\label{thm-HS-2}
    There is no complete, properly immersed, $n$-dimensional, translating soliton of the $r$-mean curvature flow $\Sigma^n\subset\R^{n+1}$, contained in the intersection of two transversal vertical half-spaces, such that $P_{r-1}\geq \ve I,$ for some $\ve>0,$ and 
    \begin{equation}\label{EquationCM}
\limsup_{\delta(x)\to\infty}\frac{\sigma_{r-1}(x)}{[\delta(x)]^2\log(\delta(x))\log(\log(\delta(x)))}<\infty,
    \end{equation}
    where $\delta:\Sigma^n\to\R_+$ is the extrinsic distance to a fixed point. \end{theorem}

\begin{remark}
    We say that $P_{r-1}\geq\ve I,$ for some $\ve>0,$ if $\lan P_{r-1}(v),v\ran\geq\ve\|v\|^2,$ for any $v\in T\Sigma^n.$
\end{remark}

In Theorem \ref{thm-HS-2}, if we take $r=1$, then $\sigma_0=1$ and $P_0=I$. Moreover, \eqref{EquationCM} is automatically satisfied. Therefore, we immediately obtain:

\begin{corollary}[Chini and M\o ller \cite{CM}] There does not exist any complete, properly immersed, 
$n$-dimensional translating soliton of the mean curvature flow, $\Sigma^n\subset\R^{n+1}$ that is contained in the intersection of two transversal vertical half-spaces.
\end{corollary}

\begin{remark}
   Clearly, the bowl soliton, the translating catenoids, and their counterparts for $r>1,$ classified in \cite{LP}, intersect every vertical hyperplane and therefore cannot be contained in the intersection of two transversal vertical half-spaces. On the other hand, when $r=1$, the Grim Reaper cylinder lies between two parallel hyperplanes; however, once one of these hyperplanes is fixed, the cylinder intersects every other vertical hyperplane transversal to it. This shows that the Grim Reaper cylinder is not contained in the intersection of two transversal vertical half-spaces either.
\end{remark}

\section{Omori-Yau type maximum principles}

The celebrated  Omori-Yau maximum principle can be considered in a variety of differential operators acting on smooth functions of a Riemannian manifold $\Sigma^n$ other than the Laplacian. In the following, we use the maximum principle found in \cite{BP} which we include a complete proof here (with more details) for the sake of completeness.

Let $\Sigma^n$ be a $n$-dimensional Riemannian manifold, $f:\Sigma^n\to\R$ be a class $\CC^2$ function, and $\phi:T\Sigma^n\to T\Sigma^n$ be a linear symmetric tensor. Define the second-order differential operator
\[
\mathcal{L}_{\phi} f := \tr(\phi\circ\hesss f) - \lan Z,\n f\ran,
\] 
where $Z$ is a vector field defined on $\Sigma^n$ with $\sup_{\Sigma^n}\| Z \| <\infty$. Here, $\hesss f:T\Sigma^n\to T\Sigma^n$ is the linear operator $\hesss f(W)=\n_W\n f,$ associated to the hessian quadratic form $\hess f$, i.e., $\hess f(W_1,W_2)=\lan\hesss f(W_1),W_2\ran.$

\begin{lemma}[G. P. Bessa and L. Pessoa, \cite{BP}]\label{BL}
Let $\Sigma^n$ be an $n$-\!\! dimensional complete Riemannian manifold, $\phi:T\Sigma^n\to T\Sigma^n$ be a symmetric and positive semidefinite linear tensor, and $Z$ be a bounded vector field on $\Sigma^n$. If there exists a positive function $\gamma\in \CC^2(\Sigma^n)$ and $G:[0,\infty)\to[0,\infty)$ such that
\begin{itemize}
\item[(i)] $G(0)>0,$ $G'(t)\geq 0,$ $G(t)^{-1/2}\not\in L^1([0,\infty));$
\item[] 
\item[(ii)] $\gamma(x)\to\infty$ when $x\to\infty;$
\item[]
\item[(iii)] $\exists A>0$ such that $\displaystyle{\|\n \gamma\|\leq A\sqrt{G(\gamma)}\left(\int_a^\gamma \frac{ds}{\sqrt{G(s)}}+1\right)}$ off a compact set, for some $a\gg1$ ;
\item[]
\item[(iv)] $\exists B>0$ such that $\displaystyle{\tr(\phi\circ\hesss \gamma) \leq B\sqrt{G(\gamma)}\left(\int_a^\gamma \frac{ds}{\sqrt{G(s)}}+1\right)}$ off a compact set, for some $a\gg1$;
\end{itemize}
\vspace{2mm}
then, for every function $u\in\CC^2(\Sigma^n)$ satisfying
\begin{equation}\label{hyp-u}
\lim_{x\to\infty}\frac{u(x)}{\vp(\gamma(x))}=0, \quad \mbox{for}\quad \vp(t)=\ln\left(\int_0^t \frac{ds}{\sqrt{G(s)}}+1\right),\ t\in[0,\infty),
\end{equation}
there exists a sequence of points $\{x_k\}_k\subset\Sigma^n$ such that
\begin{equation}\label{eq-OY}
\|\n u(x_k)\|<\frac{1}{k} \quad \mbox{and} \quad  \mathcal{L}_{\phi} u(x_k)<\frac{1}{k}.
\end{equation}
Moreover, if instead of \eqref{hyp-u} we assume that $u^\ast=\sup_{\Sigma^n} u<\infty$, then
\[
\lim_{k\to\infty} u(x_k) = u^\ast.
\] 
\end{lemma}
\begin{proof}
Let
\[
f_k(x)=u(x)-\ve_k\vp(\gamma(x)),
\]
for each positive integer $k,$ where $\ve_k>0$ is a sequence satisfying $\ve_k\to 0,$ when $k\to\infty.$ Since, for a fixed $x_0\in\Sigma^n,$ the sequence $\{f_k(x_0)\}_k$, defined by $f_k(x_0)=u(x_0)-\ve_k\vp(\gamma(x_0))$ is bounded, adding a positive constant to the function $u,$ if necessary, we may assume that $f_k(x_0)>0$ for every $k>0$. Notice that, by \eqref{hyp-u}, 
\[
\lim_{x\to\infty}\frac{f_k(x)}{\vp(\gamma(x))}=\lim_{x\to\infty}\frac{u(x)}{\vp(\gamma(x))}-\ve_k=-\ve_k< 0,
\] 
which implies that $f_k$ is non-positive out of a compact set $\Omega_{k}\subset\Sigma^n$ containing $x_0$. Thus, $f_k$ achieves its maximum at a point $x_k\in \Omega_{k}$ for each $k\geq 1$. Now, notice that
\begin{equation}\label{vp-prime}
\vp'(t)=\left[\sqrt{G(t)}\left(\int_0^t \frac{ds}{\sqrt{G(s)}}+1\right)\right]^{-1}>0
\end{equation}
and
\begin{equation}
\begin{aligned}
\vp''(t)&=-\left[\frac{G'(t)}{2\sqrt{G(t)}}\left(\int_0^t\frac{ds}{\sqrt{G(s)}}+1\right)+1\right]\times\\
&\quad\times\left[\sqrt{G(t)}\left(\int_0^t \frac{ds}{\sqrt{G(s)}}+1\right)\right]^{-2}<0.
\end{aligned}
\end{equation}
Since
\[
\n f_k = \n u - \ve_k\vp'(\gamma)\n\gamma
\]
and
\[
\hess f_k (W,W)= \hess u (W,W) - \ve_k\left[\vp'(\gamma)\hess\gamma(W,W)+\vp''(\gamma)\lan W,\n\gamma\ran^2\right],
\]
we have, at $x_k,$ that
\begin{equation}\label{eqPeng1}
\n u(x_k)=\ve_k\vp'(\gamma(x_k))\n\gamma(x_k)
\end{equation}
and
\begin{equation}\label{eqPeng2}
\begin{aligned}
\hess u(x_k)(W,W)&\leq \ve_k\vp'(\gamma(x_k))\hess\gamma(x_k)(W,W)\\
&\quad+\ve_k\vp''(\gamma(x_k))\lan W,\n\gamma(x_k)\ran^2\\
&\leq \ve_k\vp'(\gamma(x_k))\hess\gamma(x_k)(W,W).
\end{aligned}
\end{equation}
Now, Equation \eqref{eqPeng1} and hypothesis (iii) imply
\[
\|\n u(x_k)\|=\ve_k\vp'(x_k)\|\n\gamma(x_k)\|\leq \ve_k A <\frac{1}{k},
\] 
for $\ve_k<\frac{1}{kA}.$ On the other hand, letting   $\{e_1,\ldots,e_n\}$ be  an orthonormal frame formed with  eigenvectors of $\phi: T\Sigma^n\to T\Sigma^n,$ with  nonnegative eigenvalues $\lambda_1,\ldots,\lambda_n,$  we have, using \eqref{vp-prime} and hypothesis (iii) and (iv),
\[
\aligned
\mathcal{L}_{\phi}  u(x_k)&= \sum_{i=1}^n\lan\hesss u(x_k)(e_i),\phi(e_i)\ran - \lan Z(x_k),\n u(x_k)\ran\\
&= \sum_{i=1}^n\lambda_i\lan\hesss u(x_k)(e_i),e_i\ran- \lan Z(x_k),\n u(x_k)\ran\\
& = \sum_{i=1}^n\lambda_i\hess u(x_k)(e_i,e_i)- \lan Z(x_k),\n u(x_k)\ran\\
&\leq \ve_k\vp'(\gamma(x_k))\sum_{i=1}^n\lambda_i\hess\gamma(x_k)(e_i,e_i)\\
&\quad- \ve_k\vp'(\gamma(x_k))\lan Z(x_k),\n \gamma(x_k)\ran\\
&=\ve_k\vp'(\gamma(x_k))\tr(\phi\circ\hesss \gamma)(x_k)\\
&\quad -\ve_k\vp'(\gamma(x_k))\lan Z(x_k),\n \gamma(x_k)\ran\\
&\leq \ve_k\left(B+A\sup_{\Sigma^n}\|Z\|\right)<\frac{1}{k},
\endaligned
\]
if we take 
\[
\ve_k<\frac{1}{k\max\{A,B+A\sup_{\Sigma^n}\|Z\|\}}.
\] 
\vspace{2mm}
If $=u^\ast=\sup_{\Sigma^n}u(x)<\infty,$ then, given an arbitrary integer $m>0,$ let $y_m\in\Sigma^n$ such that
\[
u(y_m)>u^\ast - \frac{1}{2m}.
\]
This gives
\[
\aligned
f_k(x_k)&=u(x_k)-\ve_k\gamma(x_k)\geq f_k(y_m)\\
&=u(y_m)-\ve_k\gamma(y_m)\\
&>u^\ast - \frac{1}{2m}-\ve_k\gamma(y_m),
\endaligned
\]
which implies
\[
\aligned
u(x_k)&>u^\ast - \frac{1}{2m}-\ve_k\gamma(y_m)+\ve_k\gamma(x_k)\\
&>u^\ast - \frac{1}{2m}-\ve_k\gamma(y_m).\\
\endaligned
\]
Now, choosing $k_m$ such that $\ve_{k_m}\gamma(y_m)<\frac{1}{2m},$ we obtain that
\[
u(x_{k_m})>u^\ast -\frac{1}{m}.
\]
Thus, by replacing $x_k$ by $x_{k_m}$ if necessary, we can conclude that 
\[
\lim_{k\to\infty}u(x_k)=u^\ast.
\]
\end{proof}

\begin{remark}\label{rem-ln}
The typical examples of functions $G$ satisfying condition (i) of Lemma \ref{BL} are given by
\begin{equation}\label{example-G}
    G(t)=t^2\prod_{j=1}^N\left(\log^{(j)}(t)\right)^2, \quad t\gg1,
\end{equation}
where $\log^{(j)}(t)$ denotes the $j$-th iterate of $\log t.$
\end{remark}
In the following, we apply Lemma \ref{BL} to the context of isometric immersions. Let $X:\Sigma^n\to \R^{n+1}$ be an isometric immersion and $P_r$ the $r$-th Newton transformation defined by \eqref{Pr-defi}, p.\pageref{Pr-defi}. We introduce the functional operator
\[
L_{r-1}f=\tr(P_{r-1}\circ\hesss f), \quad f\in \CC^2(\Sigma^n).
\]
This operator is important in the study of $\sigma_r$ as a generalization of the Laplacian, see, for example \cite{ABS-2023}, \cite{ABS-2025}, \cite{AdCE-2003}, \cite{AdCR-1993}, \cite{AIR-2013}, \cite{BX-2024}, \cite{BMR-2013}, \cite{CY-1977}, \cite{HL-1995}, \cite{IMR-2011}, \cite{R-1993}, and \cite{W-1985}.

Applying Lemma \ref{BL} to the operator $L_{r-1}$, we obtain

\begin{theorem}\label{theo-max-princ-Lr}
Let $\Sigma^n$ be a complete hypersurface of $\R^{n+1}$ and $A:T\Sigma^n\to T\Sigma^n$ be its shape operator. Assume that the $(r-1)$-th Newton transformation $P_{r-1}$ is positive semidefinite. If one of the following conditions holds:
    \begin{itemize}
        \item[(i)] $\sigma_{r-1}$ is bounded and
        \begin{equation}\label{hyp-OY-1}
        \limsup_{\rho(x)\to\infty}\frac{\|A(x)\|}{\rho(x)\log(\rho(x))\log(\log(\rho(x)))}<\infty,    
        \end{equation}
where $\rho:\Sigma^n\to\R_+$ is the intrinsic distance on $\Sigma^n$ to a fixed point;
        \item[(ii)] or $\Sigma^n$ is properly immersed, 
        \begin{equation}\label{hyp-OY-2}
\limsup_{\delta(x)\to\infty}\frac{|\sigma_{r-1}(x)|}{[\delta(x)]^2\log(\delta(x))\log(\log(\delta(x)))}<\infty,
        \end{equation}
and 
\begin{equation}\label{hyp-OY-3}
\limsup_{\delta(x)\to\infty}\frac{|\sigma_{r}(x)|}{\delta(x)\log(\delta(x))\log(\log(\delta(x)))}<\infty,
\end{equation}
where $\delta:\Sigma^n\to\R_+$ is the extrinsic distance of $\R^{n+1}$ to a fixed point, restricted to $\Sigma^n,$
    \end{itemize}
then, for any class $\CC^2$ function $f:\Sigma^n\to\R,$ bounded from above, there exists a sequence of points $\{x_k\}_k\subset\Sigma^n$ such that
\begin{equation}\label{CY-max-princ-0}
\begin{cases}
\displaystyle{\lim_{k\to \infty} f(x_k)}=\sup_{\Sigma^n} f,\\
\|\n f (x_k)\|\leq \dfrac{1}{k},\\
L_{r-1} f (x_k) \leq \dfrac{1}{k}.\\
\end{cases}
\end{equation}
\end{theorem}
\begin{proof}
To prove the first part, under the hypothesis (i), let us follow Example 1.13 of \cite{prs}. Let $\gamma(x)=\rho(x)^2=[\dist_{\Sigma}(x,p_0)]^2,$ where $\dist_\Sigma(x,p_0)$ is the intrinsic distance of $\Sigma^n$ to a point $p_0\in\Sigma^n.$ Then $\gamma$ is smooth in $\Sigma^n\setminus(\{p_0\}\cup\mbox{cut}(p_0)),$ where $\mbox{cut}(p_0)$ denotes the cut locus of $p_0.$ Since, for the points at the cut locus, we can apply the Calabi's trick, we will work only with the points where $\gamma$ is smooth. 

If $\|A(x)\|^2\leq C G_0(\rho(x)),$ for some smooth function $G_0:[0,+\infty)\to\R,$ even at the origin, then, by the Gauss equation, the sectional curvatures $K_{\Sigma}$ of $\Sigma^n$ satisfies
\[
K_{\Sigma}\geq -2\|A\|^2\geq - 2CG_0(\rho).
\]
This implies, following the proof of Example 1.13 of \cite{prs} step-by step, that
\[
\hess\gamma(Y,Y)\leq B\gamma^{1/2}G_0(\gamma^{1/2})^{1/2}\|Y\|^2
\]
for some $B>0$ and $\rho(x)$ sufficiently large. If $\lambda_1,\ldots,\lambda_n$ are the eigenvalues of $P_{r-1}$ with eigenvetors $e_1,\ldots,e_n,$ we have
\begin{equation}\label{L-gamma}
\begin{aligned}
L_{r-1}\gamma &= \tr(P_{r-1}\circ\hesss \gamma)\\ 
&= \sum_{j=1}^{n}\langle  P_{r-1}(\hesss \gamma(e_j)),e_j\rangle\\
&= \sum_{j=1}^{n}\langle \hesss\gamma(e_j), P_{r-1}(e_j)\rangle\\
&= \sum_{j=1}^{n}\lambda_j\langle \hesss\gamma(e_j), (e_j)\rangle\\
&=\sum_{j=1}^n\lambda_j\hess\gamma(e_j,e_j)\\
&\leq B(\tr P_{r-1})\gamma^{1/2}G_0(\gamma^{1/2})^{1/2}.
\end{aligned}
\end{equation}
Now, taking $G(t)=(t\log t)^2$ and $\gamma=\rho^2,$ in order to apply Lemma \ref{BL}, observe that
\begin{equation}\label{P-gamma}
\begin{aligned}
P(\gamma)&:=\sqrt{G(\gamma)}\left(\int_a^\gamma \frac{ds}{\sqrt{G(s)}}+1\right)\\
&=\gamma\log(\gamma)[\log(\log(\gamma))-\log(\log(a))+1]\\
&=2\rho^2\log\rho[\log(\log(2\rho) - \log(\log(a))+1]\\
&=2\rho^2\log\rho[\log(\log(\rho)+\log 2 - \log(\log(a))+1]\\
&=2\rho^2\log\rho\log(\log\rho),
\end{aligned}
\end{equation}
for $a=e^{2e}.$ Clearly $\|\n\gamma\|=2\rho\leq P(\gamma)$ for $\rho$ sufficiently large. On the other hand, the condition (iv) of Lemma \ref{BL} holds if we assume that $\sigma_{r-1}$ is bounded and take 
\begin{equation}\label{G0}
G_0(t)=(t\log t\log(\log t))^2.
\end{equation}
Indeed, since $\tr(P_{r-1})=(n-r+1)\sigma_{r-1}$ (see Lemma 2.1, p.279 of \cite{BC}), Equation \eqref{L-gamma} gives 
\[
L_{r-1}\gamma \leq B_1\sigma_{r-1}\rho G_0(\rho)^{1/2} \leq B_2\rho^2\log\rho\log(\log\rho).
\]
In order to prove the second part, i.e., under the hypothesis (ii), we will follow the ideas of Example 1.14 of \cite{prs}. Indeed, assume that $\Sigma^n$ is properly immersed and let $\gamma(x)=[\delta(x)]^2=\|X(x)-X(p_0)\|^2$. By using \eqref{P-gamma} with $\delta$ in the place of $\rho,$ we have $\|\n\gamma\|\leq 2\delta\leq P(\gamma)$ for large values of $\delta$. On the other hand,
\begin{equation}\label{hess-delta-2}
\begin{aligned}
\hess\gamma(Y,Z) &= \overline{\hess}\ \delta^2(Y,Z) + \lan A(X),Y\ran\lan N,\overline{\n}\delta^2\ran\\
&=2\lan X,Y\ran + \lan A(X),Y\ran\lan N,\overline{\n}\delta^2\ran,\\
\end{aligned}
\end{equation}
where $\overline{\n}$ and $\overline{\hess}$ denote the gradient and the hessian of $\R^{n+1},$ respectively. This implies
\begin{equation}\label{Lr-delta-2}
\begin{aligned}
L_{r-1}\gamma &= 2(\tr P_{r-1}) + 2(\tr(A\circ P_{r-1}))\lan N,\delta\overline{\n}\delta\ran,\\
&=2(n-r+1)\sigma_{r-1}+2r\sigma_r\lan N,\delta\overline{\n}\delta\ran,\\
\end{aligned}
\end{equation}
where, in the last equality, we used that
\begin{equation}
    \tr(P_{r-1})=(n-r+1)\sigma_{r-1} \quad \mbox{and} \quad \tr(A\circ P_{r-1})=r\sigma_r,
\end{equation}
(see Lemma 2.1, p.279 of \cite{BC}). If 
\[
|\sigma_{r-1}|\leq B_1 \gamma^{1/2}G_0(\gamma^{1/2})^{1/2} \quad \mbox{and} \quad |\sigma_r|\leq B_2 G_0(\gamma^{1/2})^{1/2}
\]
for large $\delta(x),$ and for $G_0(t)$ given by \eqref{G0}, then
\[
\begin{aligned}
|L_{r-1}\gamma|&\leq 2(n-r+1)|\sigma_{r-1}|+2r|\sigma_r|\gamma^{1/2}\\
&\leq 2(n-r+1)B_1 \gamma^{1/2}G_0(\gamma^{1/2})^{1/2} + 2r B_2\gamma^{1/2}G_0(\gamma^{1/2})^{1/2}\\
&=:B_3\gamma^{1/2}G_0(\gamma^{1/2})^{1/2}\\
&\leq B_4 P(\gamma).
\end{aligned}
\]
The result then follows by applying Lemma \ref{BL} for $G(t)=(t\log t)^2.$
\end{proof}
\begin{remark}\label{rem-ln-2}
Clearly, Theorem \ref{theo-max-princ-Lr} holds replacing, $G(t)=(t\log t)^2$ by those given in Equation \eqref{example-G} of Remark \ref{rem-ln}. Our choice was aesthetic.
\end{remark}

\section{proof of Theorem \ref{thm-Cone}}
\begin{proof}[Proof of Theorem \ref{thm-Cone}]
Let us denote the extrinsic distance to the origin by $\delta_0(x)=\|X(x)\|$.
    From $[\delta_0(x)]^2=\lan X(x),X(x)\ran$ we have
    \begin{equation}\label{n-delta-0}    
        \delta_0\n\delta_0 = X^\top\quad \mbox{and}\quad \|\n\delta_0\| = \frac{\|X^\top\|}{\|X\|}\leq 1.
    \end{equation}
    
    On the other hand, by \eqref{Lr-delta-2} and by the definition of $L_{r-1},$
    \begin{equation}\label{Lr1}
        \frac{1}{2}L_{r-1}\delta_0^2=(n-r-1)\sigma_{r-1}+r\sigma_r\lan X,N\ran
    \end{equation}
    and
    \begin{equation}\label{Lr2}
        \frac{1}{2}L_{r-1}\delta_0^2=\delta_0 L_r\delta_0 + \lan P_{r-1}(\n\delta_0),\n\delta_0\ran.
    \end{equation}
    Combining Equations \eqref{n-delta-0}, \eqref{Lr1}, and \eqref{Lr2} gives
    \begin{equation}\label{Lr3}
        L_{r-1}\delta_0 = \frac{1}{\delta_0}[(n-r+1)\sigma_{r-1}+r\sigma_r\lan X,N\ran] -\frac{1}{\delta_0^3}\lan P_{r-1}(X^\top),X^\top\ran.
        \end{equation}
        Suppose by contradiction that a translating soliton $\Sigma^{n}\subset \mathbb{R}^{n+1}$  satisfying the hypotheses of Theorem \ref{thm-Cone} is contained in the complement of a cone $\CC_{V,a}$, this is  
\[\Sigma^n\subset \left(\CC_{V,a}\right)^c=\left\{X\in\R^{n+1}; \,\left\lan \frac{X}{\|X\|},V\right\ran \leq \,a, \,\,a\in (0,1)\right\}.\]
    Define the function $\psi: \Sigma^n\to\R$ by
    \begin{equation}
        \psi(x):=\lan X(x),V\ran - a\|X(x)\|\leq 0.
    \end{equation}
Since for any $U_1,U_2\in T\Sigma^n,$
\begin{equation}\label{deriv-XV}
    U_1(\lan X,V\ran)=\lan U_1,V\ran \quad \mbox{and} \quad U_2U_1(\lan X,V\ran)=\lan\overline{\n}_{U_2}U_1,V\ran,
\end{equation}
we have $
\n\lan X,V\ran = V^\top$. This gives 
\[
\n\psi = V^\top -a\n\delta_0,
\]
and then, by \eqref{n-delta-0},
\[
\|\n\psi\|=\|V^\top -a\n\delta_0\|\geq\left\vert \|V^\top\|-a\frac{\|X^\top\|}{\|X\|}\right\vert.
\]
Now, choosing an orthonormal frame $\{e_1,e_2,\ldots,e_n\},$ defined on $\Sigma^n,$ formed with the eigenvectors of $P_{r-1}$ corresponding to eigenvalues $\lambda_{i}$, we have that
\[
\begin{aligned}
\hess\lan X,V\ran (e_i,e_j)&= e_ie_j(\lan X,V\ran) - \n_{e_i}e_j(\lan X,V\ran)\\
&=\lan\overline{\n}_{e_i}e_j,V\ran - \lan\n_{e_i}e_j,V\ran\\
&=\lan B(e_i,e_j),V\ran\\
&=\lan A(e_i),e_j\ran\lan N,V\ran. 
\end{aligned}
\]
Recalling that (see Lemma 2.1, p.279 of \cite{BC}) \begin{equation}
    \tr(P_{r-1})=(n-r+1)\sigma_{r-1} \quad \mbox{and} \quad \tr(A\circ P_{r-1})=r\sigma_r,
\end{equation}
  we have
\begin{equation}\label{Lr-XV}    
\begin{aligned}
L_{r-1}\lan X,V\ran &= \tr(P_{r-1}\circ\hesss \lan X,V\ran)\\
&= \sum_{i=1}^{n}\langle  P_{r-1}(\hesss \lan X,V\ran(e_i)),e_i\rangle\\
&= \sum_{i=1}^{n}\langle \hesss\lan X,V\ran(e_i), P_{r-1}(e_i)\rangle\\
&= \sum_{i=1}^{n}\lambda_{i}\langle \hesss\lan X,V\ran(e_i), (e_i)\rangle\\
&= \sum_{i=1}^{n}\lambda_{i} \hess\lan X,V\ran(e_i,e_i)\\
&= \sum_{i=1}^{n}\lambda_{i}\lan A(e_i),e_i\ran\lan N,V\ran\\
&= \sum_{i=1}^{n}\lan A(e_i),P_{r-1}(e_i)\ran\lan N,V\ran\\
&= \sum_{i=1}^{n}\lan e_i,A P_{r-1}(e_i)\ran\lan N,V\ran\\
&= {\rm trace}(A\circ P_{r-1})\lan N,V\ran\\
&=  r\sigma_r\lan N,V\ran.
\end{aligned}
\end{equation}

This implies,
    \[
    \begin{aligned}
    L_{r-1}\psi &= L_{r-1}\lan X,V\ran - aL_{r-1}\delta_0 \\ &= r\sigma_r\lan N,V\ran -aL_{r-1}\delta_0 \\ &= r\sigma_r^2 - aL_{r-1}\delta_0.
    \end{aligned}
    \]
    Since $\psi$ is bounded from above, and assuming either condition (i) or (ii) in the statement of Theorem \ref{thm-Cone}, we may apply Theorem \ref{theo-max-princ-Lr}, p.\pageref{theo-max-princ-Lr}. Indeed, if \eqref{HS2-hyp-1} or \eqref{HS2-hyp-2} hold, then \eqref{hyp-OY-1} or \eqref{hyp-OY-2} follows, respectively, while \eqref{hyp-OY-3} is automatically satisfied since $|\sigma_r|\leq 1$ for translating solitons. Therefore, there exists a sequence of points $\{x_k\}_k\subset\Sigma^n$ such that
    \begin{equation}\label{OY2-1}
        \lim_{k\to\infty}\psi(x_k)=\sup_{\Sigma^n} \psi \leq 0,
    \end{equation}
    \begin{equation}\label{OY2-2}
        \frac{1}{k}>\|\n\psi (x_k)\|\geq \left|\|V^\top(x_k)\|-a\frac{\|x_k^\top\|}{\|x_k\|}\right|,
    \end{equation}
    and
    \begin{equation}\label{OY2-3}
   \frac{1}{k}>L_{r-1}\psi(x_k)=r[\sigma_r(x_k)]^2 -aL_{r-1}\delta_0(x_k).
    \end{equation}
    By \eqref{OY2-2}, it holds
    \begin{equation}\label{eq-HS-1}
        \|V^\top(x_k)\|-\frac{1}{k}\leq a\frac{\|x_k^\top\|}{\|x_k\|}<\frac{1}{k}+\|V^\top(x_k)\|.
    \end{equation}
    If $\sigma_r(x_k)=\lan N(x_k),V\ran\to0$, when $k\to\infty$, then $\|V^\top(x_k)\|\to1.$ Thus, by \eqref{eq-HS-1},
    \[
    \lim_{k\to\infty}\frac{\|x_k^\top\|}{\|x_k\|}=\frac{1}{a}>1,
    \]
which is an absurd. 

In the general case, since $|\sigma_r|=|\lan N,V\ran|\leq 1,$ there exists a subsequence $\{x_{k_l}\}$ such that $\sigma_r(x_{k_l})$ converges for $l\to\infty.$ This implies that
\[
\|V^\top(x_{k_l})\|^2=1-\lan V,N(x_{k_l})\ran^2 = 1-[\sigma_r(x_{k_l})]^2
\]
converges, and  by \eqref{eq-HS-1},  the sequence $\|x_{k_l}^\top\|/\|x_{k_l}\|$ also converges. Therefore,
\[
\begin{aligned}
\lim_{l\to\infty}\left(1-[\sigma_r(x_{k_l})]^2\right)&=\lim_{l\to\infty}\left(1-\lan V,N(x_{k_l})\ran^2\right)\\
&=\lim_{l\to\infty} \|V^\top(x_{k_l})\|^2\\
&= a^2\lim_{l\to\infty}\frac{\|x_{k_l}^\top\|^2}{\|x_{k_l}\|^2}\\
&=a^2\lim_{l\to\infty}\left(1-\frac{\lan x_{k_l},N(x_{k_l})\ran^2}{\|x_{k_l}\|^2}\right).
\end{aligned}
\]
This implies
\[
\lim_{l\to\infty}[\sigma_r(x_{k_l})]^2\left(1-[\sigma_r(x_{k_l})]^2\right)=a^2\lim_{l\to\infty}[\sigma_r(x_{k_l})]^2\left(1-\frac{\lan x_{k_l},N(x_{k_l})\ran^2}{\|x_{k_l}\|^2}\right),
\]
i.e.,
\begin{equation}\label{comp-Sr}
\begin{aligned}
\lim_{l\to\infty}&\left([\sigma_r(x_{k_l})]^4 - (1-a^2)[\sigma_r(x_{k_l})]^2\right)\\
&=a^2\lim_{l\to\infty}[\sigma_r(x_{k_l})]^2\left\lan\frac{x_{k_l}}{\|x_{k_l}\|},N(x_{k_l})\right\ran^2.\end{aligned}
\end{equation}
In particular,
\[
\lim_{l\to\infty}\left([\sigma_r(x_{k_l})]^2-(1-a^2)\right)\geq0.
    \]
On the other hand, by \eqref{OY2-3} and \eqref{Lr3}, we have 
\[
\begin{aligned}
\frac{1}{k}&>L_{r-1}\psi(x_k)\\
&=r[\sigma_r(x_k)]^2 - aL_{r-1}\|x_k\|\\
&=r[\sigma_r(x_k)]^2
-\frac{a}{\|x_k\|}[(n-r+1)\sigma_{r-1}(x_k)+r\sigma_r(x_k)\lan x_k,N(x_k)\ran]\\
&\quad+\frac{1}{\|x_k\|^3}\lan P_{r-1}x_k^\top,x_k^\top\ran\\
&\geq r[\sigma_r(x_k)]^2
-\frac{a}{\|x_k\|}[(n-r+1)\sigma_{r-1}(x_k)+r\sigma_r(x_k)\lan x_k,N(x_k)\ran]\\
&\geq  r[\sigma_r(x_k)]^2 - a(n-r+1)\frac{\sigma_{r-1}(x_k)}{\|x_k\|}
-ar\left\vert\sigma_r(x_k)\left\lan\frac{x_k}{\|x_k\|},N(x_k)\right\ran\right\vert.
\end{aligned}
\]
This implies
\begin{equation}\label{comp-Sr-1}
\begin{aligned}
\frac{1}{k}+a(n-r+1)\frac{\sigma_{r-1}(x_k)}{\|x_k\|}&>r[\sigma_r(x_k)]^2\\
&\quad-ar\left|\sigma_r(x_k)\left\lan\frac{x_k}{\|x_k\|},N(x_k)\right\ran\right|. 
\end{aligned}
\end{equation}
Thus, using \eqref{comp-Sr} and passing to the subsequence $x_{k_l}$, we obtain
\begin{equation}\label{comp-Sr-2}
\begin{aligned}
    \frac{1}{r}&\lim_{l\to\infty}\left[\frac{1}{k_l}+a(n-r+1)\frac{\sigma_{r-1}(x_{k_l})}{\|x_{k_l}\|}\right]\\
    &\geq\lim_{l\to\infty} [\sigma_r(x_{k_l})]^2-\sqrt{[\sigma_r(x_{k_l})]^4-(1-a^2)[\sigma_r(x_{k_l})]^2}\\
    &=\lim_{l\to\infty}\alpha\left([\sigma_r(x_{k_l})]^2\right),
\end{aligned}
\end{equation}
where $\alpha:[1-a^2,1]\to\R$ is defined by 
\[
\alpha(t)=t-\sqrt{t^2-(1-a^2)t}.
\]
Since $\alpha'(t)<0$, $\alpha$ is a decreasing function with minimum value at $t=1$, $\alpha(1)=1-a.$ 

In order to conclude the proof, let us consider separately the cases of conditions (i) and (ii) of Theorem \ref{thm-Cone}. First, let us assume condition (ii) of Theorem \ref{thm-Cone}. This implies that
$\sigma_{r-1}$ is bounded and $\Sigma^n$ is not necessarily proper.

Observe that the region $(\CC_{V,a})^c$ is invariant by translations in the direction $-V,$ i.e., if $\Sigma^n\subset (\CC_{V,a})^c,$ then $(\Sigma^n)_t:=\Sigma^n - tV\subset (\CC_{V,a})^c$ for any $t>0$. Moreover, since $\Sigma^n$ is a translating soliton with velocity vector $V$, each $(\Sigma^n)_t$ is also a translating soliton with the same velocity vector $V$. Then, considering $\Sigma^n$ as one of the $(\Sigma^n)_t$, if necessary, we can assume that
\[
\inf_{\Sigma^n}\|x\|\geq \frac{2a(n-r+1)\sup_{\Sigma^n}\sigma_{r-1}}{r(1-a)}.
\]
This gives
\[
\begin{aligned}
0&\geq \lim_{l\to\infty} \alpha\left(\sigma_r(x_{k_l})^2\right) - \frac{1}{r}\lim_{l\to\infty}\left[\frac{1}{k_l}+a(n-r+1)\frac{\sigma_{r-1}(x_{k_l})}{\|x_{k_l}\|}\right]\\
&\geq 1-a - \frac{1}{r}\lim_{l\to\infty}\left[\frac{1}{k_l}+a(n-r+1)\frac{\sup_{\Sigma^n}\sigma_{r-1}}{\inf_{\Sigma^n}\|x\|}\right]\\
&\geq \frac{1-a}{2}\\
&>0\\
\end{aligned}
\]
which is an absurd. 

Now, assume that condition (i) of Theorem \ref{thm-Cone} holds, i.e., $\Sigma^n$ is properly immersed and
\[
\limsup_{\delta_0(x)\to\infty}\frac{\sigma_{r-1}(x)}{\delta_0(x)}<\frac{r(1-a)}{a(n-r+1)},
\]
(the difference between $\delta$ and $\delta_0$ is a constant). This gives
\[
\begin{aligned}
0&\geq \lim_{l\to\infty} \alpha\left(\sigma_r(x_{k_l})^2\right) - \frac{1}{r}\lim_{l\to\infty}\left[\frac{1}{k_l}+a(n-r+1)\frac{\sigma_{r-1}(x_{k_l})}{\|x_{k_l}\|}\right]\\
&\geq 1-a - \frac{1}{r}\lim_{l\to\infty}\left[\frac{1}{k_l}+a(n-r+1)\frac{\sigma_{r-1}(x_{k_l})}{\|x_{k_l}\|}\right]\\
&>0\\
\end{aligned}
\]
and we obtain an absurd again.
\end{proof}

\section{Proof of Theorem \ref{thm-HS-1}}

\begin{proof}[Proof of Theorem \ref{thm-HS-1}.] Let $\Sigma^n \subset \mathbb{R}^{n+1}$ be a properly immersed translating soliton of the $r$-mean curvature flow, with velocity $V$ and operator $P_{r-1}$ positive semidefinite satisfying \eqref{HS1-hyp-1}. Suppose  $\Sigma^n\subset\mathcal{H}_W,$ where
     \[
     \mathcal{H}_{W}=\{X\in \mathbb{R}^{n+1}; \langle X, W\rangle\leq0,\ \langle V, W\rangle>0,\ \|W\|=1\}.
     \]
Let $c:=\lan V,W\ran>0$ and define $\psi:\Sigma^n\to\R$ by $\psi(x)=\lan X(x),W\ran$. By using \eqref{deriv-XV} and \eqref{Lr-XV} with $W$ in the place of $V,$ we have
\[
\n\psi = W^\top
\]
and
\[
L_{r-1}\psi = r\sigma_r\lan N,W\ran.
\]
Since, by hypothesis, $\psi\leq 0,$ by Theorem \ref{theo-max-princ-Lr}, p.\pageref{theo-max-princ-Lr}, there exists a sequence of points $\{x_k\}_k\subset\Sigma^n$ such that 
\begin{equation}\label{OY-1}
\lim_{k\to\infty}\psi(x)=\sup_{\Sigma^n}\psi\leq 0,
\end{equation}
\begin{equation}\label{OY-2}
\displaystyle\frac{1}{k}>\|\n\psi(x_k)\|^2=\|W(x_k)^\top\|^2=1-\lan W,N(x_k)\ran^2 \geq 0,
\end{equation}
and
\begin{equation}\label{OY-3}
\displaystyle\frac{1}{k}> L_{r-1}\psi(x_k)=r\sigma_r(x_k)\lan N(x_k),W\ran.
\end{equation}
Thus
\begin{equation}\label{limit-1}
   \lim_{k\to\infty}\lan W,N(x_k)\ran^2=1.
\end{equation}
Since, for any orthonormal frame $\{e_1,\ldots,e_n\}$ on $\Sigma^n,$ 
\[
W=\sum_{i=1}^{n} \langle W, e_i\rangle e_i + \langle W, N\rangle N,
\] 
we have, by \eqref{limit-1}, that
\[
\lim_{k\to\infty}\sum_{i=1}^n\lan W,e_i(x_k)\ran^2=0.
\]
Moreover, 
\[
\begin{aligned}
\langle V, W\rangle &=\left\langle \sum_{i=1}^{n}\langle V, e_i\rangle e_i+ \langle V, N\rangle N,\sum_{j=1}^{n}\langle W, e_j\rangle e_j + \langle W, N\rangle N\right\rangle\\
&=\sum_{i=1}^{n}\langle V, e_i\rangle \langle W, e_i\rangle + \langle V,N \rangle \langle W, N\rangle
\end{aligned}
\]
and the estimate $|\lan V,e_i(x_k)\ran|\leq \|V\|\|e_i(x_k)\|\leq 1,$ imply
\begin{equation}\label{limit-2}
\begin{aligned}
0&\geq \lim_{k\to\infty} L_{r-1}\psi(x_k)\\
&=r\lim_{k\to\infty}\lan N(x_k),V\ran\lan N(x_k),W\ran\\
&=r\left[\lan V,W\ran - \lim_{k\to\infty}\sum_{i=1}^n\lan V,e_i(x_k)\ran\lan e_i(x_k),W\ran\right]\\
&= r\lan V,W\ran =r\cdot c>0,
\end{aligned}
\end{equation}
which is a contradiction.
\end{proof}

\section{Proof of Theorem \ref{thm-HS-2}}
\begin{proof}[Proof of Theorem \ref{thm-HS-2}]
This proof is inspired by  Borbely's proof (see \cite{B}) that complete minimal  surfaces satisfying the Omori-Yau maximum principle cannot be in the intersection of two transversal vertical half-spaces.

Let $\Sigma^{n}\subset \mathbb{R}^{n+1}$ be a translating soliton of the $r$-mean curvature flow with velocity $V$ and $P_{r-1}\geq \ve I$, $\ve >0$. Suppose, in addition, that $\sigma_{r-1}$ satisfies \eqref{EquationCM}. We are going to show that $\Sigma^{n}$ cannot be contained in the intersection of two transversal vertical half-spaces. 

Let
\[
\HH_i:=\HH_{(B_i,W_i)}:=\{X\in\R^{n+1};\lan X-B_i,W_i\ran\geq0\},\quad i=1,2,
\]
be two transversal vertical halfspaces. In order to simplify the proof, we can choose a system of coordinates of $\R^{n+1}$ such that 
\[
V=E_{n+1}=(0,\ldots,0,1),
\] 
$B_i=(0,\ldots,0)$, $i=1,2$ and, by a rotation, we may consider
\[
W_1=(a,b,0,\ldots,0), \quad W_2=(a,-b,0,\ldots,0),
\]
where $a,b>0$ and $a^2+b^2=1.$ In this system of coordinates, we have that
\[
\PP_i:=\partial \HH_i = \{X\in\R^{n+1};\lan X,W_i\ran=0\},\quad i=1,2.
\]
Denote by $\PP_i(R):=\PP_i+RW_i=\{X\in\R^{n+1};\lan X,W_i\ran=R\}$  the hyperplanes parallel to $\PP_i$ and by $\LL(R)$ be the $(n-1)$-hyperplane $\LL(R)=\PP_1(R)\cap \PP_2(R).$ Since the coordinates of each hyperplane $\PP_i(R),$ $i=1,2,$ satisfy the equation 
\[
ax_1+(-1)^{i-1}bx_2=R,
\]
we have 
\[
\LL(R)=\left\{\left(\frac{R}{a},0,x_3,\ldots,x_{n+1}\right); (x_3,\ldots,x_{n+1})\in\R^{n-1}\right\}.
\]
Define 
\begin{equation}\label{dR}
    d_R(x):=\mbox{dist}_{\R^{n+1}}(x,\LL(R))=\sqrt{\left(x_1-\frac{R}{a}\right)^2+x_2^2},
\end{equation}
for $x=(x_1,x_2,\ldots,x_{n+1})\in\R^{n+1},$ and consider the cylindrical region
\[
\mathcal{D}(R)=\{x\in\R^{n+1}; d_R(x)\leq R\}.
\]

 %

The cylindrical region $\DD(R)$ separates $(\HH_1\cap \HH_2)\setminus\DD(R)$ into two regions, one with $d_R(x)$ bounded (which we denote by $\VV_R$) and another where $d_R(x)$ is unbounded (see Figure \ref{fig:1}). 
\begin{figure}[ht]
\centering
\def\r{5}
	\begin{tikzpicture}[scale=0.8,>=stealth]
		\def\u{8}
		\def\v{5}
		\coordinate (p1) at (30:8cm);
		\coordinate (p2) at (-30:8cm);
		\fill[black!10] (p1) -- (0,0) -- (p2) -- cycle;
		\draw[fill=white] (4,0) circle (2cm) node[below right]{};
		\fill (4,0) circle (2pt);

		\node (k1) at (13.1:4.6cm) {$\mathcal{D}(R)$};		
		\node (k2) at (-13:4.4cm) {$\stackrel{\nwarrow}{{\rm radius}}\! R$};
		\node (k3) at (4:5cm) {$\mathcal{L}(R)$};
		\node[above] (k4) at (30:4cm) {$\mathcal{P}_2$};
		\node[below] (k5) at (-30:4cm) {$\mathcal{P}_1$};
		
		\path[draw,-{Stealth[length=2mm]}] (-1,0) -- (\u,0) node[pos=1, below] {$x_{1}$};
		\path[draw,-{Stealth[length=2mm]}] (0,-\v) -- (0,\v) node[pos=1, left] {$x_{2}$};
		
		\draw[->,shorten <=12pt,shorten >=3pt] (4:5cm)to[out=180,in=60] (4,0);
		
		\def\t{2}
		\coordinate (w1) at (60:1.6*\t cm);
		\coordinate (w2) at (-60:1.6*\t cm);
		
		\path[draw,-{Stealth[length=2mm]}] (60:-0.6*\t cm) -- (w1) node[pos=1, above] {$W_{1}$};
		\path[draw,-{Stealth[length=2mm]}] (-60:-0.6*\t cm) -- (w2) node[pos=1, below] {$W_{2}$};
		
		\draw[dashed,draw=black] (0,0) -- (0,0|-w2) node[pos=1, left] {$-b$};
		\draw[dashed] (0,0) -- (0,0|-w1) node[pos=1, left] {$b$};
		
		\draw (p1) -- (0,0) (p2) -- (0,0);
		
		\draw[dashed] (w1) -- (0,0|-w1) (w1) -- (w1|-0,0) node[below right] {};

		\draw[dashed] (w2) -- (0,0|-w2) (w2) -- (w2|-0,0);
		
		\node[above] (t1) at (0.6*\t,0) {$\mathcal{V}_{R}$};
		
		\draw[dashed] (4,0) -- +(120:2cm) (4,0) -- +(240:2cm);

	\end{tikzpicture}
     \caption{Regions in the intersection of vertical halfspaces}
    \label{fig:1}
    
    \end{figure}
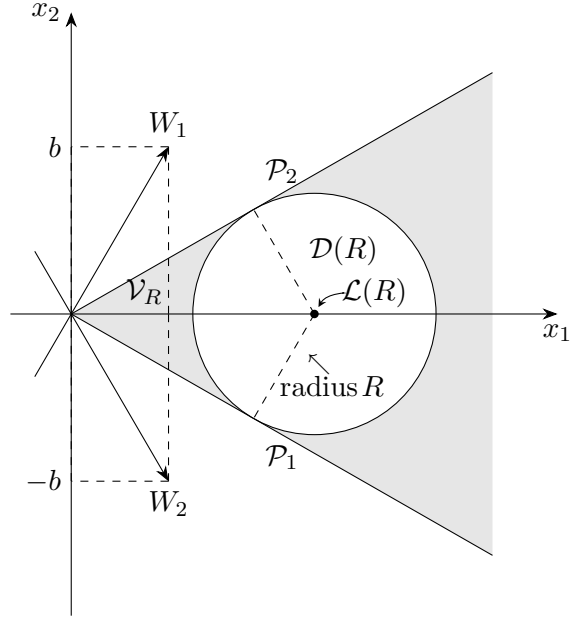

We assume, by contradiction, that there exists a complete, properly immersed, translating soliton of the $r$-mean curvature flow, with $P_{r-1}\geq \ve I,$ for $\ve>0$, such that $\sigma_{r-1}$ satisfies \eqref{HS2-hyp-2} and $\Sigma^n\subset\HH_1\cap \HH_2$. 

Since $\VV_R\to \HH_1\cap \HH_2$ when $R\to\infty,$ we can assume $R>0$ large enough such that $\Sigma^n\cap \VV_R\neq\emptyset.$ Let us denote by $\overline{\n}$ the gradient and connection of $\R^{n+1}$ and by $\overline{\hess}$ the hessian quadratic form of $\R^{n+1}.$ The following facts are straightforward calculations about the function $d(x):=d_R(x),$ for $x\in(\HH_1\cap \HH_2)\backslash\mathcal{L}(R):$

\begin{itemize}
    \item[(i)] $\overline{\hess}\ d(x)(\overline{\n}d,\overline{\n}d)=0$ which implies that $\overline{\n}d$ is an eigenvector of $\overline{\hess}\ d$ with eigenvalue zero;
    \item[(ii)] Since $d(x)$ does not depend on $x_3,\ldots,x_{n+1},$ the vectors $E_3,\ldots,E_{n+1}$ of the canonical basis of $\R^{n+1}$ are also eigenvectors of $\overline{\hess}\ d$ with eigenvalue zero;
    \item[(iii)] The last eigenvector of $\overline{\hess}\ d$ is given by
    \[
    \chi=\left(-\frac{\partial d}{\partial x_2},\frac{\partial d}{\partial x_1},\vec{\bf{0}}\right)=\left(-\frac{x_2}{d},\frac{x_1-R/a}{d},\vec{\bf{0}}\right),\quad \vec{\bf{0}}\in\R^{n-1}.
    \]
    For this eigenvector, the eigenvalue is $1/d.$ 
\end{itemize}
These facts imply that the set $\{\overline{\n}d,\chi,E_3,\ldots,E_{n+1}\}$ is an orthonormal frame of $\R^{n+1}$ given by eigenvectors of $\overline{\hess}\ d.$ This implies that any vector field $Y\in\R^{n+1}$ can be written as
\begin{equation}\label{Y-Eq}
Y=\lan Y,\overline{\n}d\ran\overline{\n}d+\lan Y,\chi\ran\chi + \sum_{j=3}^{n+1}\lan Y,E_j\ran E_j.
\end{equation}
Moreover, for any $Y,Z\in\R^{n+1},$
\begin{equation}\label{hess-d-1}
    \overline{\hess}\ d(Y,Z)=\frac{1}{d}\lan Y,\chi\ran\lan Z,\chi\ran.
\end{equation}
Now, let $f:\Sigma^n\to\R$ be defined by
\[
f(x)=\left\{
\begin{aligned}
    d_R(x),&\quad \mbox{if}\quad x\in\Sigma^n\cap\VV_R,\\
    R,&\quad \mbox{if}\quad x\in\Sigma^n\backslash\VV_R.
\end{aligned}
\right.
\]
Notice that $f$ is continuous and differentiable everywhere, except for points of $\Sigma^n\cap(\partial\VV_R\cap\partial\DD(R)).$ Moreover $R\leq f(x)<R/a$ (since $0<a<1$) and
\[
R<\sup_{\Sigma^n} f \leq R/a<\infty,
    \]
since $\Sigma^n\cap\VV_R\neq\emptyset$ and $f(x)>R$ for $x\in\Sigma^n\cap\VV_R$ (see Figure \ref{fig:1}). The following calculations will be carried out at points $x \in \Sigma^n \cap \VV_R$ where the function $d_R$ is differentiable. Since
\begin{equation}\label{grad-f-1}
    \n f = (\overline{\n}d)^\top=\overline{\n}d-(\overline{\n}d)^\perp = \overline{\n}d - \lan\overline{\n}d,N\ran N
\end{equation}
and, by replacing $N$ in the place of $Y$ in \eqref{Y-Eq},
\[
\begin{aligned}
1&=\lan\chi,N\ran^2+\lan\overline{\n}d,N\ran^2 + \sum_{j=3}^{n+1}\lan E_j,N\ran^2\\
&\geq \lan\chi,N\ran^2+\lan\overline{\n}d,N\ran^2,
\end{aligned}
\]
we have
\begin{equation}\label{grad-f-2}
    \|\n f\|=\sqrt{\|\overline{\n}d\|^2-\lan\overline{\n}d,N\ran^2}=\sqrt{1-\lan\overline{\n}d,N\ran^2}\geq|\lan\chi,N\ran|,
\end{equation}
provided $\|\overline{\n}d\|=1.$

If $\{e_1,e_2,\ldots,e_n\}$ is an orthonormal frame of $\Sigma^n,$ given by eigenvectors of the second fundamental form, then 
\begin{equation}\label{hess-d-2}
\begin{aligned}
\sum_{i=1}^n\overline{\hess}\ d(e_i,P_{r-1}(e_i)) &= \frac{1}{d}\sum_{i=1}^n\lan e_i,\chi\ran\lan P_{r-1}(e_i),\chi\ran\\
&=\frac{1}{d}\sum_{i=1}^n\lambda_i\lan e_i,\chi\ran^2,\\
\end{aligned}
\end{equation}
where $\lambda_i$ is the eigenvalue of $P_{r-1}$ relative to the eigenvector $e_i,$ $i=1,\ldots,n.$ Since, by hypothesis, $\lambda_i\geq \ve>0$ and
\[
\sum_{i=1}^n\lan e_i,\chi\ran^2 = 1-\lan N,\chi\ran^2,
\]
we have
\begin{equation}\label{hess-d-3}
\sum_{i=1}^n\overline{\hess}\ d(e_i,P_{r-1}(e_i)) \geq \ve\left(\frac{1-\lan N,\chi\ran^2}{d}\right).
\end{equation}
On the other hand, for any $Y,Z\in T\Sigma^n,$ and for points $x\in\Sigma^n\cap\VV_R,$ 
\[
\begin{aligned}
\hess\ f (Y,Z)=\overline{\hess}\ d(Y,Z) + \lan \overline{\n}d,N\ran\lan A(X),Y\ran, 
\end{aligned}
\]
where $A$ is the shape operator of $\Sigma^n.$
This gives, for any orthonormal frame $\{e_1,e_2,\ldots,e_n\}$ of $T\Sigma^n,$
\begin{equation}\label{eq-Lr-trans-1}
   \begin{aligned}
    L_{r-1}f&=\tr(\hesss f\circ P_{r-1})=\sum_{i=1}^n\hess f(e_i,P_{r-1}(e_i))\\
    &=\sum_{i=1}^n\overline{\hess}\ d(e_i,P_{r-1}(e_i))+\lan \overline{\n}d,N\ran\sum_{i=1}^n\lan A(e_i),P_{r-1}(e_i)\ran\\
    &=\sum_{i=1}^n\overline{\hess}\ d(e_i,P_{r-1}(e_i))+r\sigma_r\lan \overline{\n}d,N\ran\\
    &=\sum_{i=1}^n\overline{\hess}\ d(e_i,P_{r-1}(e_i))+r\lan E_{n+1},N\ran\lan \overline{\n}d,N\ran,\\
    \end{aligned}
\end{equation}
where we used that 
\[
\sum_{i=1}^n\lan A(e_i),P_{r-1}(e_i)\ran=\tr (AP_{r-1})=r\sigma_r
\]
and that $\sigma_r=\lan E_{n+1},N\ran.$ Combining \eqref{eq-Lr-trans-1} with \eqref{hess-d-3}, we obtain
\begin{equation}\label{eq-Lr-trans-2}
L_{r-1}f \geq \ve\left(\frac{1-\lan\chi,N\ran^2}{d}\right) + r\lan E_{n+1},N\ran\lan \overline{\n}d,N\ran.
\end{equation}
Since $f$ is bounded, $f(x)=R$ for $x\in\Sigma^n\backslash\VV_R$ and $\sup_{\Sigma^n} f>R,$ by Theorem \ref{theo-max-princ-Lr}, p.\pageref{theo-max-princ-Lr}, there exists a sequence $\{x_k\}_k$ of points in $\Sigma^n\cap\VV_R$ such that
\begin{equation}\label{Omori-Yau-Trans-1}    \lim_{k\to\infty}f(x_k)=\sup_{\Sigma^n} f, \lim_{k\to\infty}\|\n f(x_k)\|=0 \quad \mbox{and}\quad \limsup_{k\to\infty} L_{r-1}f(x_k)\leq 0.
\end{equation}
Now, let us analyze the last term of \eqref{eq-Lr-trans-2}. First notice that, since 
\[
\lan E_{n+1},\overline{\n}d\ran=0,
\] 
we have
\begin{equation}\label{EN}
\begin{aligned}
\lan E_{n+1},N\ran^2 &= \lan E_{n+1},N\pm\overline{\n}d\ran^2\\
&\leq\|N\pm\overline{\n}d\|^2\\
&=2(1\pm\lan N,\overline{\n}d\ran).  
\end{aligned}
\end{equation}
Using \eqref{grad-f-2} and the second limit in \eqref{Omori-Yau-Trans-1}, we obtain that
\begin{equation}\label{limit-Nd}
    \lim_{k\to\infty}\lan N(x_k),\overline{\n}d(x_k)\ran^2=1.
\end{equation}
It means that there exists at least one subsequence (which we also denote by $\{x_k\}_k$ to simplify the notation) such that
\[
\lim_{k\to\infty}\lan N(x_k),\overline{\n}d(x_k)\ran=1 \quad \mbox{or}\quad \lim_{k\to\infty}\lan N(x_k),\overline{\n}d(x_k)\ran=-1.
\]
In both cases, we can use \eqref{EN} to obtain 
\begin{equation}\label{limit-EN}
    \lim_{k\to\infty}\lan E_{n+1},N(x_k)\ran=0.
\end{equation}
This implies
\begin{equation}\label{limit-EN-2}
    \lim_{k\to\infty}\lan E_{n+1},N(x_k)\ran\lan N(x_k),\overline{\n}d(x_k)\ran=0.
\end{equation}
On the other hand, by \eqref{grad-f-2} and \eqref{Omori-Yau-Trans-1}, we have
\begin{equation}\label{lim-chi}
\lim_{k\to\infty}\lan N(x_k),\chi(x_k)\ran=0. 
\end{equation}
Thus, by using \eqref{eq-Lr-trans-2}, \eqref{Omori-Yau-Trans-1}, and \eqref{lim-chi}, we have
\[
\begin{aligned}
0&\geq \limsup_{k\to\infty} L_{r-1}f(x_k)\\
&\geq\ve\limsup_{k\to\infty}\left(\frac{1-\lan N(x_k),\chi(x_k)\ran^2}{d(x_k)}\right)\\
&\quad +r\limsup_{k\to\infty}\lan E_{n+1},N(x_k)\ran\lan N(x_k),\overline{\n}d(x_k)\ran\\
&\geq \frac{\ve a}{R}>0,
\end{aligned}
\]
which is a contradiction. Here we used that $d(x)\leq R/a.$ Thus, there is no such translating soliton.
\end{proof}

\section*{Acknowledgments}
The authors were partially supported by the National Council for Scientific and Technological Development - CNPq of Brazil (Grant number 303118/2022-9 to H. Alencar, 303057/2018-1, 402563/2023-9 to G. P. Bessa and 306189/2023-2 to G. Silva Neto).


\begin{bibdiv}
\begin{biblist}


\bib{ABS-2023}{article}{
 author={Alencar, Hil{\'a}rio},
 author={Batista, M{\'a}rcio},
 author={Silva Neto, Greg{\'o}rio},
 issn={0022-0396},
 issn={1090-2732},
 doi={10.1016/j.jde.2023.05.045},
 review={Zbl 1521.53029},
 title={Poincar{\'e} type inequality for hypersurfaces and rigidity results},
 journal={Journal of Differential Equations},
 volume={369},
 pages={156--179},
 date={2023},
 publisher={Elsevier (Academic Press), San Diego, CA},
}

\bib{ABS-2025}{article}{
 author={Alencar, Hil{\'a}rio},
 author={Bessa, G. Pacelli},
 author={Silva Neto, Greg{\'o}rio},
 issn={0022-1236},
 issn={1096-0783},
 doi={10.1016/j.jfa.2025.110920},
 review={Zbl 08017507},
 title={Gap theorems for complete self-shrinkers of {{\(r\)}}-mean curvature flows},
 journal={Journal of Functional Analysis},
 volume={289},
 number={1},
 pages={21},
 note={Id/No 110920},
 date={2025},
 publisher={Elsevier, Amsterdam},
}

\bib{AdCE-2003}{article}{
 author={Alencar, Hil{\'a}rio},
 author={do Carmo, Manfredo},
 author={Elbert, Maria Fernanda},
 issn={0075-4102},
 issn={1435-5345},
 doi={10.1515/crll.2003.006},
 review={Zbl 1093.53063},
 title={Stability of hypersurfaces with vanishing {{\(r\)}}-mean curvatures in Euclidean spaces},
 journal={Journal f{\"u}r die Reine und Angewandte Mathematik},
 volume={554},
 pages={201--216},
 date={2003},
 publisher={De Gruyter, Berlin},
}



\bib{AdCR-1993}{article}{
 author={Alencar, Hil{\'a}rio},
 author={do Carmo, Manfredo P.},
 author={Rosenberg, Harold},
 issn={0232-704X},
 issn={1572-9060},
 doi={10.1007/BF00773553},
 review={Zbl 0816.53031},
 title={On the first eigenvalue of the linearized operator of the {{\(r\)}}-th mean curvature of a hypersurface},
 journal={Annals of Global Analysis and Geometry},
 volume={11},
 number={4},
 pages={387--395},
 date={1993},
 publisher={Springer Netherlands, Dordrecht},
}


\bib{AS2010}{article}{
author={Alessandroni, Roberta},
author={Sinestrari, Carlo},
title={Evolution of hypersurfaces by powers of the scalar curvature},
journal={Ann. Sc. Norm. Super. Pisa Cl. Sci. (5)},
volume={9},
date={2010},
number={3},
pages={541--571},
issn={0391-173X},
review={\MR{2722655}},
}

\bib{AIR-2013}{article}{
 author={Al{\'{\i}}as, Luis J.},
 author={Impera, Debora},
 author={Rigoli, Marco},
 issn={0002-9947},
 issn={1088-6850},
 doi={10.1090/S0002-9947-2012-05774-6},
 review={Zbl 1276.53064},
 title={Hypersurfaces of constant higher order mean curvature in warped products},
 journal={Transactions of the American Mathematical Society},
 volume={365},
 number={2},
 pages={591--621},
 date={2013},
 publisher={American Mathematical Society (AMS), Providence, RI},
}

\bib{AMC2012}{article}{
author={Andrews, Ben},
author={McCoy, James},
title={Convex hypersurfaces with pinched principal curvatures and flow of
convex hypersurfaces by high powers of curvature},
journal={Trans. Amer. Math. Soc.},
volume={364},
date={2012},
number={7},
pages={3427--3447},
issn={0002-9947},
review={\MR{2901219}},
doi={10.1090/S0002-9947-2012-05375-X},
}

\bib{AW2021}{article}{
author={Andrews, Ben},
author={Wei, Yong},
title={Volume preserving flow by powers of the $k$-th mean curvature},
journal={J. Differential Geom.},
volume={117},
date={2021},
number={2},
pages={193--222},
issn={0022-040X},
review={\MR{4214340}},
doi={10.4310/jdg/1612975015},
}

\bib{BC}{article}{
 author={Barbosa, Jo{\~a}o Lucas Marques},
 author={Colares, Ant{\^o}nio Gervasio},
 issn={0232-704X},
 issn={1572-9060},
 doi={10.1023/A:1006514303828},
 review={Zbl 0891.53044},
 title={Stability of hypersurfaces with constant {{\(r\)}}-mean curvature},
 journal={Annals of Global Analysis and Geometry},
 volume={15},
 number={3},
 pages={277--297},
 date={1997},
 publisher={Springer Netherlands, Dordrecht},
}

\bib{BX-2024}{article}{
 author={Batista, M{\'a}rcio},
 author={Xavier, Wagner},
 issn={0219-1997},
 issn={1793-6683},
 doi={10.1142/S0219199723500232},
 review={Zbl 1542.53100},
 title={On the rigidity of self-shrinkers of the {{\(r\)}}-mean curvature flow},
 journal={Communications in Contemporary Mathematics},
 volume={26},
 number={6},
 pages={15},
 note={Id/No 2350023},
 date={2024},
 publisher={World Scientific, Singapore},
}

\bib{BS2018}{article}{
author={Bertini, Maria Chiara},
author={Sinestrari, Carlo},
title={Volume preserving flow by powers of symmetric polynomials in the
principal curvatures},
journal={Math. Z.},
volume={289},
date={2018},
number={3-4},
pages={1219--1236},
issn={0025-5874},
review={\MR{3830246}},
doi={10.1007/s00209-017-1995-8},
}

\bib{BP}{article}{
 author={Bessa, G. Pacelli},
 author={Pessoa, Leandro F.},
 issn={1678-7544},
 issn={1678-7714},
 doi={10.1007/s00574-014-0047-9},
 review={Zbl 1297.53038},
 title={Maximum principle for semi-elliptic trace operators and geometric applications},
 journal={Bulletin of the Brazilian Mathematical Society. New Series},
 volume={45},
 number={2},
 pages={243--265},
 date={2014},
 publisher={Springer, Berlin/Heidelberg; Sociedade Brasileira de Matem{\'a}tica (SBM), Rio de Janeiro},
}

\bib{BMR-2013}{book}{
 author={Bianchini, Bruno},
 author={Mari, Luciano},
 author={Rigoli, Marco},
 isbn={978-0-8218-8799-8},
 isbn={978-1-4704-1056-8},
 issn={0065-9266},
 issn={1947-6221},
 book={
 title={On some aspects of oscillation theory and geometry},
 publisher={Providence, RI: American Mathematical Society (AMS)},
 },
 doi={10.1090/S0065-9266-2012-00681-2},
 review={Zbl 1295.34047},
 title={On some aspects of oscillation theory and geometry},
 series={Memoirs of the American Mathematical Society},
 volume={1056},
 pages={v + 195},
 date={2013},
 publisher={American Mathematical Society (AMS), Providence, RI},
}


\bib{B}{article}{
 author={Borb{\'e}ly, Albert},
 issn={0004-9727},
 doi={10.1017/S0004972711002346},
 review={Zbl 1220.53071},
 title={On minimal surfaces satisfying the Omori-Yau principle},
 journal={Bulletin of the Australian Mathematical Society},
 volume={84},
 number={1},
 pages={33--39},
 date={2011},
 publisher={Cambridge University Press, Cambridge},
}

\bib{CRS2010}{article}{
author={Cabezas-Rivas, Esther},
author={Sinestrari, Carlo},
title={Volume-preserving flow by powers of the $m$th mean curvature},
journal={Calc. Var. Partial Differential Equations},
volume={38},
date={2010},
number={3-4},
pages={441--469},
issn={0944-2669},
review={\MR{2647128}},
doi={10.1007/s00526-009-0294-6},
}

\bib{CY-1977}{article}{
 author={Cheng, Shiu-Yuen},
 author={Yau, Shing-Tung},
 issn={0025-5831},
 issn={1432-1807},
 doi={10.1007/BF01425237},
 review={Zbl 0349.53041},
 title={Hypersurfaces with constant scalar curvature},
 journal={Mathematische Annalen},
 volume={225},
 pages={195--204},
 date={1977},
 publisher={Springer, Berlin/Heidelberg},
 eprint={https://eudml.org/doc/162932},
}

\bib{CM}{article}{
 author={Chini, Francesco},
 author={M{\o}ller, Niels Martin},
 issn={1073-7928},
 issn={1687-0247},
 doi={10.1093/imrn/rnz183},
 review={Zbl 1506.53095},
 title={Bi-halfspace and convex hull theorems for translating solitons},
 journal={IMRN. International Mathematics Research Notices},
 volume={2021},
 number={17},
 pages={13011--13045},
 date={2021},
 publisher={Oxford University Press, Cary, NC},
}

\bib{Chow1987}{article}{
author={Chow, Bennett},
title={Deforming convex hypersurfaces by the square root of the scalar
curvature},
journal={Invent. Math.},
volume={87},
date={1987},
number={1},
pages={63--82},
issn={0020-9910},
review={\MR{0862712}},
doi={10.1007/BF01389153},
}

\bib{CSS}{article}{
 author={Clutterbuck, Julie},
 author={Schn{\"u}rer, Oliver C.},
 author={Schulze, Felix},
 issn={0944-2669},
 issn={1432-0835},
 doi={10.1007/s00526-006-0033-1},
 review={Zbl 1120.53041},
 title={Stability of translating solutions to mean curvature flow},
 journal={Calculus of Variations and Partial Differential Equations},
 volume={29},
 number={3},
 pages={281--293},
 date={2007},
 publisher={Springer, Berlin/Heidelberg},
}

\bib{GLM2018}{article}{
author={Gao, Shanze},
author={Li, Haizhong},
author={Ma, Hui},
title={Uniqueness of closed self-similar solutions to
$\sigma_k^{\alpha}$-curvature flow},
journal={NoDEA Nonlinear Differential Equations Appl.},
volume={25},
date={2018},
number={5},
pages={Paper No. 45, 26},
issn={1021-9722},
review={\MR{3845754}},
doi={10.1007/s00030-018-0535-5},
}

\bib{GLW2017}{article}{
author={Guo, Shunzi},
author={Li, Guanghan},
author={Wu, Chuanxi},
title={Volume-preserving flow by powers of the $m$-th mean curvature in
the hyperbolic space},
journal={Comm. Anal. Geom.},
volume={25},
date={2017},
number={2},
pages={321--372},
issn={1019-8385},
review={\MR{3690244}},
doi={10.4310/CAG.2017.v25.n2.a3},
}

\bib{HIMW}{article}{
 author={Hoffman, David},
 author={Ilmanen, Tom},
 author={Mart{\'{\i}}n, Francisco},
 author={White, Brian},
 isbn={978-3-030-68540-9},
 isbn={978-3-030-68543-0},
 isbn={978-3-030-68541-6},
 book={
 title={Minimal surfaces: integrable systems and visualisation. M:iv workshops, 2016--19},
 publisher={Cham: Springer},
 },
 doi={10.1007/978-3-030-68541-6\_9},
 review={Zbl 1484.53009},
 title={Notes on translating solitons for mean curvature flow},
 pages={147--168},
 date={2021},
}

\bib{HL-1995}{article}{
 author={Hounie, Jorge},
 author={Leite, Maria Luiza},
 issn={0022-040X},
 issn={1945-743X},
 doi={10.4310/jdg/1214456216},
 review={Zbl 0821.53007},
 title={The maximum principle for hypersurfaces with vanishing curvature functions},
 journal={Journal of Differential Geometry},
 volume={41},
 number={2},
 pages={247--258},
 date={1995},
 publisher={International Press of Boston, Somerville, MA},
}

\bib{IMR-2011}{article}{
 author={Impera, Debora},
 author={Mari, Luciano},
 author={Rigoli, Marco},
 issn={0002-9939},
 issn={1088-6826},
 doi={10.1090/S0002-9939-2010-10649-4},
 review={Zbl 1220.53073},
 title={Some geometric properties of hypersurfaces with constant {{\(r\)}}-mean curvature in Euclidean space},
 journal={Proceedings of the American Mathematical Society},
 volume={139},
 number={6},
 pages={2207--2215},
 date={2011},
 publisher={American Mathematical Society (AMS), Providence, RI},
}


\bib{KP}{article}{
 author={Kim, Daehwan},
 author={Pyo, Juncheol},
 issn={2330-1511},
 doi={10.1090/bproc/67},
 review={Zbl 1472.53075},
 title={Half-space type theorem for translating solitons of the mean curvature flow in Euclidean space},
 journal={Proceedings of the American Mathematical Society. Series B},
 volume={8},
 pages={1--10},
 date={2021},
 publisher={American Mathematical Society (AMS), Providence, RI},
}

\bib{LWW2021}{article}{
author={Li, Haizhong},
author={Wang, Xianfeng},
author={Wu, Jing},
title={Contracting axially symmetric hypersurfaces by powers of the
$\sigma_k$-curvature},
journal={J. Geom. Anal.},
volume={31},
date={2021},
number={3},
pages={2656--2702},
issn={1050-6926},
review={\MR{4225822}},
doi={10.1007/s12220-020-00370-w},
}

\bib{LSW2020}{article}{
author={Li, Qi-Rui},
author={Sheng, Weimin},
author={Wang, Xu-Jia},
title={Asymptotic convergence for a class of fully nonlinear curvature
flows},
journal={J. Geom. Anal.},
volume={30},
date={2020},
number={1},
pages={834--860},
issn={1050-6926},
review={\MR{4058539}},
doi={10.1007/s12220-019-00169-4},
}

\bib{LP}{article}{
 author={de Lima, Ronaldo F.},
 author={Pipoli, Giuseppe},
 issn={1050-6926},
 issn={1559-002X},
 doi={10.1007/s12220-025-01925-5},
 review={Zbl 07990063},
 title={Translators to higher order mean curvature flows in {{\(\mathbb{R}^n \times \mathbb{R}\)}} and {{\(\mathbb{H}^n \times \mathbb{R}\)}}},
 journal={The Journal of Geometric Analysis},
 volume={35},
 number={3},
 pages={50},
 note={Id/No 92},
 date={2025},
 publisher={Springer US, New York, NY; Mathematica Josephina, St. Louis, MO},
}

\bib{prs}{book}{
 author={Pigola, Stefano},
 author={Rigoli, Marco},
 author={Setti, Alberto G.},
 isbn={978-0-8218-3639-2},
 isbn={978-1-4704-0423-9},
 issn={0065-9266},
 issn={1947-6221},
 book={
 title={Maximum principles on Riemannian manifolds and applications},
 publisher={Providence, RI: American Mathematical Society (AMS)},
 },
 doi={10.1090/memo/0822},
 review={Zbl 1075.58017},
 title={Maximum principles on Riemannian manifolds and applications},
 series={Memoirs of the American Mathematical Society},
 volume={822},
 pages={99},
 date={2005},
 publisher={American Mathematical Society (AMS), Providence, RI},
}

\bib{Reilly}{article}{
   author={Reilly, Robert C.},
   title={Variational properties of functions of the mean curvatures for
   hypersurfaces in space forms},
   journal={J. Differential Geometry},
   volume={8},
   date={1973},
   pages={465--477},
   issn={0022-040X},
   review={\MR{0341351}},
}

\bib{R-1993}{article}{
 author={Rosenberg, Harold},
 issn={0007-4497},
 review={Zbl 0787.53046},
 title={Hypersurfaces of constant curvature in space forms},
 journal={Bulletin des Sciences Math{\'e}matiques. Deuxi{\`e}me S{\'e}rie},
 volume={117},
 number={2},
 pages={211--239},
 date={1993},
 publisher={Gauthier-Villars, Paris},
}

\bib{Urbas1990}{article}{
author={Urbas, John I. E.},
title={On the expansion of starshaped hypersurfaces by symmetric
functions of their principal curvatures},
journal={Math. Z.},
volume={205},
date={1990},
number={3},
pages={355--372},
issn={0025-5874},
review={\MR{1082861}},
doi={10.1007/BF02571249},
}

\bib{Urbas1999}{article}{
author={Urbas, John},
title={Convex curves moving homothetically by negative powers of their
curvature},
journal={Asian J. Math.},
volume={3},
date={1999},
number={3},
pages={635--656},
issn={1093-6106},
review={\MR{1793674}},
doi={10.4310/AJM.1999.v3.n3.a4},
}

\bib{W-1985}{article}{
 author={Walter, Rolf},
 issn={0025-5831},
 issn={1432-1807},
 doi={10.1007/BF01455537},
 review={Zbl 0536.53054},
 title={Compact hypersurfaces with a constant higher mean curvature function},
 journal={Mathematische Annalen},
 volume={270},
 pages={125--145},
 date={1985},
 publisher={Springer, Berlin/Heidelberg},
 eprint={https://eudml.org/doc/182950},
}

\bib{Z2013}{article}{
author={Zhao, Liang},
title={The first eigenvalue of $p$-Laplace operator under powers of the
$m$th mean curvature flow},
journal={Results Math.},
volume={63},
date={2013},
number={3-4},
pages={937--948},
issn={1422-6383},
review={\MR{3057347}},
doi={10.1007/s00025-012-0242-1},
}

\end{biblist}
\end{bibdiv}
\end{document}